\documentclass[a4paper,12pt]{article}
\usepackage{array,amsfonts,amsmath,graphicx,mathrsfs,amssymb,amsthm,latexsym}
\usepackage[latin1]{inputenc}
\usepackage{color}

\newtheorem{thm}{Theorem}[section]

\newtheorem{lemma}[thm]{Lemma}

\def \cF {{\cal F}}
\def \cG {{\cal G}}
\def \cH {{\cal H}}

\begin{document}
\title{Uniformly resolvable decompositions of $K_v-I$ into $5$-stars }

\author
 {Jehyun Lee and  Melissa Keranen\\
\small Department of Mathematical Sciences \\
\small Michigan Technological University\\
}

\maketitle

\vspace{5 mm}

\begin{abstract}
We consider the existence problem of uniformly resolvable decompositions of $K_v$ into subgraphs such that each resolution class contains only blocks isomorphic to the same graph. We give a complete solution for the case in which one resolution class is $K_2$ and the rest are $K_{1,5}$.

\end{abstract}

\section{Introduction}\label{intro}



Let $G$ be a graph with vertex set $V(G)$ and edge set ${E }(G)$. An $\cH$-$decomposition$ of the graph $G$ is a collection of edge disjoint subgraphs $\cH = \{H_1,H_2, \dots, H_a \}$ such that every edge of $G$ appears in exactly one graph $H_i \in \cH$. The subgraphs, $H_i \in \cH$, are called blocks. An $\cH$-$decomposition$ is called resolvable if the blocks in $\cH$ can be partitioned into classes(or factors) $F_j$, such that every vertex of $G$ appears in exactly one block of each $F_j$. A resolvable $\cH$-$decomposition$ is also referred to as an $\cH$-$factorization$ of $G$, whose classes are referred to as $\cH$-factors. We say a class(or factor) $F$, is uniform if each $H_i \in F$ is isomorphic to a given $H$. An $\cH$-decomposition of $G$ is uniformly resolvable if its blocks can be partitioned into uniform classes. If $\cH=\{K_2 \}$, then a $K_2$-factorization of $G$ is known as a 1-{\em factorization} and its factors are called 1-{\em factors}. It is well known that a 1-factorization of $K_v$  exists if and only if $v$ is even (\cite{Lu}).

Recently, the existence problem for uniformly resolvable $\cH$-decompositions of $K_v$ have been studied, and many results have been obtained. In particular; results have been given when $\cH$ is a set of two complete graphs of order  at most five in \cite{DLD, R, SG, WG};
when
$\cH$ is a set of two or three paths on two, three or four vertices in \cite{GM1,GM2, LMT}; for
$\cH =\{P_3, K_3+e\}$ in \cite{GM}; for $\cH =\{K_3, K_{1,3}\}$ in \cite{KMT}; for $\cH =\{C_4, P_{3}\}$ in \cite{M}; for $\cH =\{K_3, P_{3}\}$ in \cite{MT}; for $\cH =\{K_2, K_{1,3}\}$ in \cite{CC}.

If $\cH=\{H_1,H_2 \}$, then there are many types of uniformly resolvable $\cH$-decompositions, depending on how many factors contain copies of $H_1$ and how many factors contain copies of $H_2$. We let $(H_1,H_2)$-$URD(v;r,s)$ denote a uniformly resolvable decomposition of $K_v$ into $r$ classes containing only copies of $H_1$ and $s$ classes containing only copies of $H_2$. We will consider this problem when $H_1 = K_2$ and $H_2 = K_{1,n}$. While the general case $(K_2,K_{1,n})$-$URD(v;r,s)$ is still open and in progression, we have observed that the standard methods used for most cases of $(r,s)$ are not applicable to solve the cases when number of $1$-factors is small. Thus, we studied these cases separately. With regard to the extremal cases, we have the following results.
\begin{itemize}
    \item A $(K_2,K_{1,n})$-$URD(v;r,0)$ exists if and only if $v$ is even.
    \item If $n$ is even, a $(K_2,K_{1,n})$-$URD(v;0,s)$ exists if and only if $v \equiv 1 \pmod{2n}$ and $v \equiv 0 \pmod{n+1}$ $(\cite{CD}, \cite{WG})$.
    \item If $n$ is odd, there exists no $(K_2,K_{1,n})$-$URD(v;0,s)$ $(\cite{CD}, \cite{WG})$.
\end{itemize}

The existence problem for a $(K_2,K_{1,3})$-$URD(v;r,s)$ for any admissible parameters $r,v,$ and $s$ has been solved in \cite{CC}. In this paper, we completely solve the existence problem of a $(K_2,K_{1,5})$-URD$(v;$ $1,\frac{v-2}{2})$ by proving the following  result.\\

\noindent \textbf{Main Theorem.} 
{\em There exists a decomposition of $K_{v}-I$ into $5$-star factors if and only if $v \equiv 12 \pmod{30}$.}

\section{Necessary Conditions}

\begin{lemma} 
\label{ness}
If a $(K_2,K_{1,5})$-$URD(v;1,s)$ exists, then $v \equiv 12 \pmod{30}$. 
\end{lemma}

\begin{proof}
Let $K_v$ be the complete graph on $v$ vertices. Because a $(K_2,K_{1,5})$-$URD(K_v;1,s)$ contains exactly one 1-factor, $v$ must be divisible by $2$. Also, if $s \geq 1$, $v$ must be divisible by $6$ and the total number of edges in $K_v-I$ must be divisible by the total number of edges in one 5-star factor. Thus, $\frac{v(v-1)}{2}-\frac{v}{2}$ must be divisible by $\frac{5v}{6}$. If we divide $\frac{v(v-2)}{2}$ by $\frac{5v}{6}$, we obtain $\frac{3(v-2)}{5}$. Since $5$ cannot divide $3$, $(v-2)$ must be divisible by $5$. Therefore, we obtain the two congruences,

\begin{align}
    v \equiv& \textrm{ } 0 \pmod{6}\\
    v \equiv& \textrm{ } 2 \pmod{5}.
\end{align}

By the Chinese remainder theorem, we have $v \equiv 12 \pmod{30}$.
\end{proof}

\section{Almost $5$-star Factors}

If $S$ is a set of $v$ vertices, such that $v \equiv t \pmod{n+1}$, then we will say the graph $G$ is {\em almost spanning} if it spans all but $t$ vertices. Define an {\em almost $n$-star factor} on a set of vertices $S$ to be an almost spanning graph on $S$ in which each connected component of $S-t$ is an $n$-star, and the $t$ isolated vertices form one $(t-1)$-star, which we will refer to as a  {\em little star}.

Let $G$ be a graph with $g$ vertices. The {\em difference} of the edge $e=\{u,v\}$ in $G$ with $u<v$, is $D(e)=min\{v-u,g-(v-u)\}$. If the difference of an edge $e$ is defined by $(v-u)$, then we will refer to this edge as a {\em forward edge}, and its difference will be called a {\em forward difference}. If the difference of an edge $e$ is defined by $g-(v-u)$, then we will refer to this edge as a {\em wrap-around edge}, and its difference will be called a {\em wrap-around difference}.

Let $F$ be an almost 5-star factor.  Label the edges in $F$ by the differences they cover. Suppose each forward difference occurs no more than twice among the stars. If any difference $d$ appears twice, then use the labels $d$(pure) and $d'$(prime) to distinguish them. Also, if $\{u,v\}$ is an edge with a prime difference, and $u<v$, then we will denote it by $\{u,v'\}$. If a star consists of edges whose differences all have a pure label, we will refer to this star as a {\em pure star}. If it consists of edges whose differences all have a prime label, then it will be referred to as a {\em prime star}. We will refer to the corresponding differences as {\em pure differences} or {\em prime differences}, and similarly; we will refer to the corresponding edges as {\em pure edges} or {\em prime edges}. If a star contains a mixture of pure edges and prime edges, then it will be referred to as a {\em mixed star}.

\subsection{t odd}
In this section, we construct {\em almost 5-star factors} when the number of isolated vertices is odd.

\begin{lemma}
\label{t1}
Let $m \geq 1$. There exists an almost 5-star factor with $t=1$ on $G=\{0,1, \ldots, 30m+6\}$ with the following properties:
\begin{itemize}
\item Each forward difference $d \in \{1, 2, \dots, 15m+3\}$ appears at least once among the stars.
\item Each forward difference $d \in \{1, 2, \dots, 15m+3\}$ appears no more than twice among the stars.
\item There are no wrap-around edges.
\item There is one mixed star.
\end{itemize}

\end{lemma}

\begin{proof}
Let $V=\{0,1, \ldots, 30m+6\}$ be a set of $30m+7$ vertices.  If $m$ is odd, we construct the following sets of pure stars, $P_1$ and $P_2$, and a mixed star, $M$ on $V \backslash \{(30m+5)\}$.

\begin{align*}
  \textrm{Let } P_1&=\{(i-1; j, j-1, j-2, j-3, j-4)\},  \\
&\mbox{where } i = 1, \ldots, \frac{5m+1}{2} \mbox{ and } j=(15m+7)-5i.
\end{align*}

For $P_2$, if $m=1$, $P_2$ is an empty set. If $m \geq 3$, then we let 

\begin{align*}
    P_2&= \{15m+3+i; j, j-6, j-12, j-18, j-24\},    \\
    &\mbox{where } i=1, 2, \ldots, \frac{m-1}{2} \mbox{ and } j=30m+18-29i.
\end{align*}

\[\textrm{Let } M=\{(15m+3); (30m+6),(30m+0), (30m-6), (30m+4)', (30m+3)'\}.\]

Let $D=\{1, 2, \ldots, 15m+3\}$  denote the pure edge set. Let $D_0^*=\{ (15m+3), (15m-3), (15m-9) \}$ and $D_1^*=\{6k | k=1, \ldots, \frac{5m-5}{2}\}$. Note that $D_1^*$ is empty if $P_2$ is empty, that is when $m=1$. Then the pure edges in $M$ exhaust the differences in $D_0^*$. Also, the stars in $P_2$ exhaust the differences in $D_1^*$, and the stars in $P_1$ exhaust the differences in $D \setminus (D^*_0 \cup D^*_1)$.

The mixed star also contains two prime edges covering the forward differences $(15m+1)'$ and $(15m)'$. We construct a set of $2m$ more prime stars $P_3$ simply by always choosing the next available smallest vertex for the center, and the set of  next available largest vertices for the leaves. In this case, the smallest possible vertex is $(15m+4)$ when $m=1$, and the largest possible vertex is $(30m+2)$. So, the maximum possible length of a prime edge among these leftover vertices is $(15m-2)$. Note that the largest available vertex $(30m+2)$ is fixed, but the value for the smallest available vertex grows as $m$ grows. Therefore, the maximum length possible for a prime edge happens when $m=1$, and it only decreases as  $m$ increases. However, at this point, the only prime edges used have differences $(15m+1)'$ and $(15m)'$. Hence, this process guarantees that all prime edges used in $P_3$ will have distinct forward differences from the set $\{(15m-1)', (15m-2)', \ldots, 1'\}$. Thus, $P_1 \cup P_2 \cup P_3$ partition $V \setminus \{(30m+5)\}$.

If $m \geq 2$ is even, then we construct the following sets of pure stars, $P_1$ and $P_2$, and a mixed star, $M$ on $V \backslash \{(30m+1)\}$. Note that the case of $m=0$ will be discussed in Lemma~\ref{t1v42}.

\begin{align*}
   \textrm{Let }  P_1&=\{i-1; j, j-1, j-2, j-3, j-4\},\\
    &\mbox{where }i=1, \ldots, \frac{5m}{2} \mbox{ and } j=15m+4-5i.
\end{align*}

If $m \leq 6$, then let 
\begin{align*}
    P_2&= \{((15m-1)+i; j,j-6,j-12,j-18,j-24)\}    \\
    &\mbox{where } i =1, \ldots, \frac{m}{2} \leq 3 \mbox{ and } j=30m+29-29i.
\end{align*}

If $m \geq 8$, then let
\begin{align*}
     P_2&= \{((15m-1)+i; j,j-6,j-12,j-18,j-24)\}   \\
    & \mbox{where } i =1, 2, 3 \mbox{ and } j=30m+29-29i. \\
     P_2&= \{((15m)+i; j,j-6,j-12,j-18,j-24)\}   \\
    &\mbox{where } i = 4, \ldots, \frac{m}{2}  \mbox{ and } j=30m+30-29i.
\end{align*}
\[\textrm{Let } M=\{(15m+3); (30m+6),(30m+5),(30m+4),(30m+3)',(30m+2)'\}.\]

Let $D=\{1,2, \ldots, 15m+3\}$, $D_0^* = \{ (15m+3), (15m+2), (15m+1) \}$, and $D_1^*= \{6k | k=1, \ldots, \frac{5m}{2}\}$. Then, the pure edges in $M$ exhaust the differences in $D_0^*$. Also, the stars in $P_2$ exhaust the differences in $D_1^*$. Thus, the stars in $P_{1}$ exhaust the differences in $D \setminus (D^*_0 \cup D^*_1)$.

The mixed edges in $M$ use the differences $(15m)'$ and $(15m-1)'$. We construct a set of $2m$ more prime stars $P_3$ simply by always choosing the next available smallest vertex for the center and the set of next available largest vertices for the leaves. In this case, the smallest possible vertex is $(15m+1)$ when $m=2$, and the largest possible vertex is $(30m-1)$. So, the maximum possible length of a prime edge among these leftover vertices is  $(15m-2)$. Note that the largest available vertex is fixed, but the smallest available vertex grows as $m$ grows. Therefore, the maximum possible length for a prime edge occurs when $m=2$ and it only decreases as $m$ increases. However, at this point, the only prime edges used have differences $(15m-1)'$ and $(15m)'$. Thus, this process guarantees that all prime edges used in $P_3$ will have distinct forward differences from the set $\{(15m-2)', (15m-3)', \ldots, 1'\}$. Thus, $P_1 \cup P_2 \cup P_3$ partition $V \setminus \{(30m+1)\}$.

In each case, we have constructed an almost 5-star factor with one isolated vertex and the desired properties.

\end{proof}


\begin{lemma}
\label{t3}
Let $m \geq 0$. There exists an almost 5-star factor with $t=3$ on $G=\{0,1, \ldots, 30m+26\}$ with the following properties:
\begin{itemize}
\item Each forward difference $d \in \{1, 2, \dots, \frac{30m+26}{2}\}$ appears at least once among the stars.
\item Each forward difference $d \in \{1, 2, \dots, \frac{30m+26}{2}\}$ appears no more than twice among the stars.
\item There is one mixed star.
\item There are no wrap-around edges.
\item The $3$ isolated vertices form a prime star of size $2$.

\end{itemize}

\end{lemma}

\begin{proof}

Let $V=\{0,1, \ldots, 30m+26\}$ be a set of $30m+27$ vertices.  If $m$ is odd, we give the following sets of pure stars  $P_0$, $P_1$ and $P_2$; prime stars $P_3$; a mixed star $M$ on $V \setminus \{(15m+9), (30m+21), (30m+20)\}\}$.

 \[\textrm{Let } P_0=\{(15m+13); (30m+26), (30m+25), (30m+24), (30m+23), (30m+22)\}.\]

\begin{align*}
    \textrm{Let } P_1&=\{(i-1; j, j-1, j-2, j-3, j-4)\}    \\
    &\mbox{where } i =1, 2, \ldots, \frac{5m+3}{2} \mbox{ and } j=15m+13-5i.
\end{align*}

If $m=1$, $P_2$ is an empty set. If $3 \leq m \leq 5$, then let
\begin{align*}
    P_2&= \{((15m+10)+i; j,j-6,j-12,j-18,j-24)\}    \\
    &\mbox{where } i =1, \ldots, \frac{m-1}{2} \leq 2 \mbox{ and } j=30m+25-29i.
\end{align*}

If $m \geq 7$, let
\begin{align*}
     P_2&= \{((15m+10)+i; j,j-6,j-12,j-18,j-24)\}   \\
    &\mbox{where } i =1, 2 \mbox{ and } j=30m+25-29i. \\
    P_2&= \{((15m+11)+i; j,j-6,j-12,j-18,j-24)\} \\
    &\mbox{where } i =3, 4, \ldots, \frac{m-1}{2}  \mbox{ and } j=30m+26-29i.
\end{align*}

 \[\textrm{Let } M=\{(15m+10); (30m+19)', (30m+18)', (30m+13), (30m+7), (30m+1)\}.\]

 Let $D=\{1, 2, \ldots, 15m+13\}$, $D^*_0=\{(15m+13),(15m+12),(15m+11),(15m+10),(15m+9), (15m+3),(15m-3),(15m-9)\}$, and $D^*_1=\{6k | k=1, \ldots, \frac{5m-5}{2}\}$. Note that $D_1^*$ is empty if $P_2$ is empty that is when $m=1$. The pure edges in $P_0$ and $M$ exhaust the differences in $D^*_0$. Also, the stars in $P_2$ exhaust the differences in $D^*_1$. Therefore, the stars in $P_{1}$ exhaust the differences in $D \setminus (D^*_0 \cup D^*_1)$.

 The mixed star also contains two prime edges covering the forward differences $(15m+9)'$ and $(15m+ 8)'$. We construct a set of $2m+1$ more prime stars $P_3$ simply by always choosing the next available smallest vertex for the center, and the set of next available largest vertices for the leaves. In this case, the smallest possible vertex is $(15m+11)$ when $m=1$, and the largest possible vertex is $(30m+17)$. So, the maximum possible length of a prime edge among these leftover vertices is $(15m+6)$. Note that the largest available vertex is fixed, but the smallest available vertex grows as $m$ grows. Therefore, the maximum length possible for a prime edge occurs when $m=1$ and it only decreases as $m$ increases. However, at this point, the only prime edges used have differences $(15m+9)'$ and $(15m+8)'$ in $M$. Thus, this process guarantees that all prime edges will have distinct forward differences from the set $\{(15m+6)', (15m+5)', \ldots, 1'\}$. Thus, $P_1 \cup P_2 \cup P_3$ partition $V \setminus \{(15m+9), (30m+21), (30m+20)\}\}$. Now let $L=\{(15m+9); (30m+21)', (30m+20)'\}$ be the prime star $L$, which contains differences $\{(15m+12)', (15m+11)'\}$.


If $m$ is even, then we give the following sets of pure stars  $P_1$ and $P_2$; prime stars $P_3$; and a mixed star $M$ on $V \setminus \{(15m+13), (30m+21), (30m+26)\}\}$.

\begin{align*}
    \textrm{Let } P_1&=\{(i-1; j, j-1, j-2, j-3, j-4)\}    \\
    &\mbox{where } i =1, 2, \ldots, \frac{5m+4}{2} \mbox{ and } j=15m+16-5i.
\end{align*}

If $m=0$, then $P_2$ is an empty set. If $m \geq 2$, then let
\begin{align*}
    P_2&= \{((15m+13)+i; j,j-6,j-12,j-18,j-24)\}    \\
    &\mbox{where } i =1, 2, \ldots, \frac{m}{2}  \mbox{ and } j=30m+43-29i
\end{align*}

\begin{align*}
     \textrm{Let } M&=\{(15m+12); (30m+22)', (30m+23)',(30m+25), (30m+24),\\
 &(30m+18)\}.
\end{align*}

 Let $D=\{1, 2, \ldots, 15m+13\}$, $D^*_0=\{(15m+13),(15m+12),(15m+6)\}$, and $D^*_1=\{6k | k=1, \ldots, \frac{5m}{2}\}$. Note that $D_1^*$ is empty if $P_2$ is empty, that is when $m=0$. Then the pure edges in $M$ exhaust the differences in $D_0^*$. Also, the stars in $P_2$ exhaust the differences in $D_1^*$, and the stars in $P_{1}$ exhaust the differences in $D \setminus (D^*_0 \cup D^*_1)$.
 
 The mixed star also contains two prime edges covering the forward differences $(15m+11)'$ and $(15m+ 10)'$. We construct a set of $2m+1$ more prime stars $P_3$ simply by always choosing the next available smallest vertex for the center, and the set of next available largest vertices for the leaves. In this case, the smallest possible vertex is $(15m+14)$ when $m=0$, and the largest possible vertex is $(30m+20)$. So, the maximum possible length of a prime edge among these leftover vertices is $(15m+6)$. Note that the largest available vertex is fixed, but the smallest available vertex grows as $m$ grows. Therefore, the maximum length possible for a prime edge occurs when $m= 0$ and it only decreases as $m$ increases. However, at this point, the only prime edges used have differences $(15m+10)'$ and $(15m+11)'$ in $M$. Thus this process guarantees that all prime edges will have distinct forward differences from the set $\{(15m+6)', (15m+5)', \ldots, 1'\}$. Thus, $P_1 \cup P_2 \cup P_3$ partition $V \setminus \{(15m+13), (30m+21), (30m+26)\}$. Now let $L=\{(15m+13), (30m+21), (30m+26)\}$ be the prime star $L$, which contains differences $\{(15m+8)', (15m+13)'\}$.

In each case, we have constructed an almost 5-star factor with a little star of size $2$ and the desired properties.

\end{proof}


\begin{lemma}
\label{t5}
Let $m \geq 1$. There exists an almost 5-star factor with $t=5$ on $G=\{0,1, \ldots, 30m+16\}$ with the following properties:

\begin{itemize}
\item Each forward difference $d \in \{1, 2, \dots, (15m+8)\}$ appears at least once among the stars.
\item Each forward difference $d \in \{1, 2, \dots, (15m+8)\}$ appears no more than twice among the stars.
\item There are no wrap-around edges.
\item There is one mixed star.
\item The $5$ isolated vertices form a prime star of size $4$.

\end{itemize}
\end{lemma}

\begin{proof}


Let $V=\{0,1, \ldots, 30m+16\}$ be a set of $30m+17$ vertices.  If $m$ is odd, we give the following sets of pure stars $P_1$ and $P_2$; prime stars $P_3$; and a mixed star $M$ on $V \setminus \{(15m+9), (30m+16), (30m+15), (30m+14), (30m+12) \}\}$.
\begin{align*}
    \textrm{Let } P_1&=\{(i-1; j, j-1, j-2, j-3, j-4)\},    \\
    &\mbox{where } i =1, 2, \ldots, \frac{5m+3}{2} \mbox{ and } j=15m+13-5i.
\end{align*}

If $m=1$, then $P_2$ is an empty set. If $m \geq 3$, then let 
\begin{align*}
    P_2&= \{((15m+10)+i; j,j-6,j-12,j-18,j-24)\}    \\
    &\mbox{where } i =1, 2, \ldots, \frac{m-1}{2}  \mbox{ and } j=30m+25-29i.
\end{align*}

 \[\textrm{Let } M=\{(15m+10); (30m+13), (30m+11)', (30m+10)', (30m+7), (30m+1)\}.\]

Let $D=\{1, 2, \ldots, 15m+13\}$, $D^*_0=\{(15m+3),(15m-3),(15m-9)\}$, and $D^*_1=\{6k | k=1, \ldots, \frac{5m-5}{2}\}$ if $m \geq 3$. Note that $D_1^*$ is empty if $P_2$ is empty, that is when $m=1$. Then the pure edges in $M$ exhaust the differences in $D_0^*$. Also, the stars in $P_2$ exhaust the differences in $D_1^*$. Hence, the stars in $P_{1}$ exhaust the differences in $D \setminus (D^*_0 \cup D^*_1)$.

The mixed edges in $M$ use the differences $(15m+1)'$ and $(15m)'$. We construct a set of $2m$ more prime stars $P_3$ simply by always choosing the next available smallest vertex for the center, and the set of next available largest vertices for the leaves. In this case, the smallest possible vertex is $(15m+11)$ when $m=1$, and the largest possible vertex is $(30m+9)$. So, the maximum possible length of a prime edge among these leftover vertices is $(15m-2)$. Note that the largest available vertex is fixed, but the smallest available vertex grows as $m$ grows. Therefore, the maximum length possible for a prime edge occurs when $m=1$ and it only decreases as $m$ increases. However, at this point, the only prime edges used have differences $(15m+1)'$ and $(15m)'$. Hence, this process guarantees that all prime edges used in $P_3$ will have distinct forward differences from the set $\{(15m-2)', (15m-3)', \ldots, 1'\}$. Thus, $P_1 \cup P_2 \cup P_3$ partition $V \setminus \{(15m+9), (30m+16), (30m+15), (30m+14), (30m+12)\}$. Let $L=\{(15m+9); (30m+16)', (30m+15)', (30m+14)', (30m+12)'\}$ be the prime star $L$, which contains differences  $\{(15m+7)', (15m+6)', (15m+5)', (15m+3)'\}$.


If $m \geq 2$ is even, then we give the following sets of pure stars $P_1$ and $P_2$; prime stars $P_3$; and a mixed star $M$ on $V \setminus \{(15m+8),(30m+16),(30m+15),(30m+9),(30m+8)\}\}$. Note that the case of $m=0$ will be discussed in Lemma~\ref{t5v102}.

\begin{align*}
    \textrm{Let } P_1&=\{(i-1; j, j-1, j-2, j-3, j-4)\}    \\
    &\mbox{where } i =1, 2, \ldots, \frac{5m+2}{2} \mbox{ and } j=15m+10-5i.
\end{align*}

If $m=2$, then let
\[P_2= \{((15m+7); (30m+7),(30m+1),(30m-5),(30m-11),(30m-17)\} \]

If, $m \geq 4$, then let

\[P_2= P_{2a}\cup P_{2b} \textrm{ where} \]
\[P_{2a} = \{((15m+7); (30m+7),(30m+1),(30m-5),(30m-11),(30m-17)\} \textrm{, and} \]
\[P_{2b} = \{((15m+7)+i; j,j-6,j-12,j-18,j-24)\} \]
\[\mbox{for } i =2, 3, \ldots, \frac{m}{2}  \mbox{ and } j=30m+37-29i.\]

 \[\textrm{Let } M=\{(15m+6); (30m+14), (30m+13), (30m+12), (30m+11)', (30m+10)'\}.\]

 Let $D=\{1, 2, \ldots, 15m+13\}$, $D^*_0=\{(15m+8), (15m+7), (15m+6)\}$, and $D^*_1=\{6k | k=1, \ldots, \frac{5m}{2}\}$. The pure edges in $M$ exhaust the differences in $D_0^*$. Also, the stars in $P_2$ exhaust the differences in $D_1^*$. Then, the stars in $P_{1}$ exhaust the differences in $D \setminus (D^*_0 \cup D^*_1)$.

The mixed star also contains two prime edges covering the forward differences $(15m+5)'$ and $(15m+ 4)'$. We construct a set of $2m$ more prime stars $P_3$ simply by always choosing the next available smallest vertex for the center, and the set of next available largest vertices for the leaves. In this case, the smallest possible vertex is $(15m+9)$ when $m=2$, and the largest possible vertex is $(30m+6)$. So, the maximum possible length of a prime edge among these leftover vertices is $(15m-3)$. Note that the largest available vertex is fixed, but the smallest available vertex grows as $m$ grows. Therefore, the maximum length possible for a prime edge occurs when $m=2$ and it only decreases as $m$ increases. However, at this point, the only prime edges used have differences $(15m+5)'$ and $(15m+4)'$ in $M$. Hence, this process guarantees that all prime edges used in $P_3$ will have distinct forward differences from the set $\{(15m-3)', (15m-2)', \ldots, 1'\}$. Thus, $P_1 \cup P_2 \cup P_3$  partition $V \setminus \{(15m+8),(30m+16),(30m+15),(30m+9),(30m+8)\}$. Let $L=\{(15m+8);(30m+16)',(30m+15)',(30m+9)',(30m+8)'\}$ be the prime star $L$, which contains differences  $\{(15m+8)', (15m+7)'\, (15m+1)'\, (15m)'\}$.

In each case, we have constructed an almost 5-star factor with a little star of size $4$ and the desired properties.

\end{proof}

\subsection{t even}
In this section, we construct {\em almost 5-star factors} when the number of isolated vertices is even.

\begin{lemma}
\label{t0}
Let $m \geq 0$. There exists an almost 5-star factor with $t=0$ on $G=\{0,1, \ldots, 30m+11\}$ with the following properties:
\begin{itemize}
\item Each forward difference $d \in \{1, 2, \dots, 15m+5\}$ appears at least once among the stars.
\item Each forward difference $d \in \{1, 2, \dots, 15m+5\}$ appears no more than twice among the stars.
\item There are no wrap-around edges.
\item There is no mixed star.
\end{itemize}

\end{lemma}

\begin{proof}
Let $V=\{0,1, \ldots, 30m+11\}$ be a set of $30m+12$ vertices.  If $m$ is odd, we give the following sets of pure stars $P_0$, $P_1$ and $P_2$; and prime stars $P_3$ on $V$.

\begin{align*}
    \textrm{Let } P_0&=\{(15m+6); (30m+11),(30m+10),(30m+9),(30m+3),\\ 
     (30m&-3) \} \textrm{ and } P_1=\{(i-1; j, j-1, j-2, j-3, j-4)\}\\
    &\mbox{for } i = 1, \ldots, \frac{5m+1}{2} \mbox{ and } j=(15m+7)-5i.
\end{align*}

If $m=1$, $P_2$ is an empty set. If $m \geq 3$, then let 
\begin{align*}
     P_2&= \{(15m+6)+i; j, j-6, j-12, j-18, j-24\},   \\
    &\mbox{for } i=1, 2, \ldots, \frac{m-1}{2} \mbox{ and } j=(30m+21)-29i.
\end{align*}

Let $D=\{1, 2, \ldots, 15m+5\}$, and $D_1^*=\{6k | k=1, \ldots, \frac{5m-5}{2}\}$. Note that $D_1^*$ is empty if $P_2$ is empty, that is when $m=1$. The stars in $P_2$ exhaust the differences in $D_1^*$. Then the stars in $P_0 \cup P_1$ exhaust the differences in $D \setminus (D^*_1)$.

Next we construct a set of $2m+1$ prime stars $P_3$, simply by always choosing the next available smallest vertex for the center, and the set of next available largest vertices for the leaves. In this case, the smallest possible vertex is $(15m+3)$ when $m=1$, and the largest possible vertex is $(30m+8)$. So, the maximum possible length of a prime edge among these leftover vertices is $(15m+5)$. Note that the largest available vertex $(30m+8)$ is fixed, but the smallest available vertex grows as $m$ grows. Therefore, the maximum length possible for a prime edge occurs when $m=1$, and it only decreases as $m$ increases. However, at this point, no prime edges have appeared yet. Hence, this process guarantees that all prime edges used in $P_3$ will have distinct forward differences from the set $\{(15m+5)', (15m+4)', \ldots, 1'\}$. Thus, $P_1 \cup P_2 \cup P_3$ partition $V$.

If $m \geq 0$ is even, we give the following sets of pure stars $P_1$ and $P_2$; and prime stars $P_3$; on $V$.
\begin{align*}
    \textrm{Let } P_1&=\{(i-1; j, j-1, j-2, j-3, j-4)\}    \\
    &\mbox{for } i = 1, \ldots, \frac{5m+2}{2} \mbox{ and } j=(15m+10)-5i
\end{align*}

If $m=0$, $P_2$ is an empty set. If $m \geq 2$, then let 

\begin{align*}
    P_2&= \{(15m+5)+i; j, j-6, j-12, j-18, j-24\}    \\
    &\mbox{for } i=1, 2, \ldots, \frac{m}{2} \mbox{ and } j=(30m+35)-29i.
\end{align*}

Let $D=\{1, 2, \ldots, 15m+5\}$, and $D_1^*=\{6k | k=1, \ldots, \frac{5m}{2}\}$. Note that $D_1^*$ is empty if $P_2$ is empty, that is when $m=0$. The stars in $P_2$ exhaust the differences in $D_1^*$. Then, the stars in $P_1$ exhaust the differences in $D \setminus (D^*_1)$.

Next, for the prime edges, we construct a set of $2m+1$  prime stars, $P_3$, simply by always choosing the next available smallest vertex for the center, and the set of next available largest vertices for the leaves. In this case, the smallest possible vertex is $(15m+6)$ when $m=0$, and the largest possible vertex is $(30m+11)$. So, the maximum possible length of a prime edge among these leftover vertices is $(15m+5)$. Note that the largest available vertex $(30m+8)$ is fixed, but the smallest available vertex $(15m+6)$ grows as $m$ grows. Therefore, the maximum length possible for a prime edge occurs when $m=0$, and it only decreases as $m$ increases. However, at this point, no prime edges have appeared yet. Hence this process guarantees that all prime edges used in $P_3$ will have distinct forward differences from the set $\{(15m+5)', (15m+4)', \ldots, 1'\}$. Thus, $P_1 \cup P_2 \cup P_3$ partition $V$.

In each case, we have constructed an almost 5-star factor with no isolated vertex and the desired properties.

\end{proof}


\begin{lemma}
\label{t2}
Let $m \geq 1$. There exists an almost 5-star factor with $t=2$ on $G=\{0,1, \ldots, 30m+1\}$ with the following properties:
\begin{itemize}
\item Each forward difference $d \in \{1, 2, \dots, 15m\}$ appears at least once among the stars.
\item Each forward difference $d \in \{1, 2, \dots, 15m\}$ appears no more than twice among the stars.
\item There are no wrap-around edges.
\item There is an edge on the two isolated vertices.
\end{itemize}

\end{lemma}

\begin{proof}
Let $V=\{0,1, \ldots, 30m+1\}$ be a set of $30m+2$ vertices.  If $m$ is odd, we give the following sets of pure stars $P_0$, $P_1$ and $P_2$; and prime stars $P_3$ on $V \setminus \{(15m-3), (30m-3)\}$.

\[\textrm{Let } P_0=\{(15m+1); (30m+1),(30m),(30m-1),(30m-2),(30m-8) \},\]

\begin{align*}
    P_1&=\{(i-1; j, j-1, j-2, j-3, j-4)\}    \\
    & \mbox{for } i = 1, \ldots, \frac{5m-1}{2} \mbox{ and } j=(15m+1)-5i.
\end{align*}

If $m=1$, $P_2$ is an empty set. If $m \geq 3$, then let 
\begin{align*}
     P_2&= \{(15m+1)+i; j, j-6, j-12, j-18, j-24\}   \\
    &\mbox{for } i=1, 2, \ldots, \frac{m-1}{2} \mbox{ and } j=(30m+16)-29i.
\end{align*}

\[\textrm{Let }L=\{(15m-3), (30m-3)\}\]

Let $D=\{1, 2, \ldots, 15m\}$, and $D_1^*=\{6k | k=1, \ldots, \frac{5m-5}{2}\}$. Note that $D_1^*$ is empty if $P_2$ is empty, that is when $m=1$. The stars in $P_2$ exhaust the differences in $D_1^*$. Then the stars in $P_1$ exhaust the differences in $D \setminus (D^*_1)$.

Next we construct a set of $2m$ prime stars $P_3$ simply by always choosing the next available smallest vertex for the center, and the set of next available largest vertices for the leaves. In this case, the smallest possible vertex is $(15m-2)$, and the largest possible vertex is $(30m-4)$. So, the maximum possible length of a prime edge among these leftover vertices is $(15m-2)$ for any $m$. However, at this point, no prime edges have appeared yet. Hence this process guarantees that all prime edges used in $P_3$ will have distinct forward differences from the set $\{(15m-1)', (15m-2)', \ldots, 1'\}$. Thus, $P_0 \cup P_1 \cup P_2 \cup P_3$ partition $V \setminus \{(15m-3), (30m-3)\}$. Now let $L = \{(15m-3), (30m-3)'\}$ be the prime star $L$, which contains difference $\{(15m)' \}$.

If $m \geq 2$ is even, we give the following sets of pure stars $P_1$ and $P_2$; and prime stars $P_3$ on $V \setminus \{(15m), (30m)\}$. Note that the case of $m=0$ will be discussed in  Lemma~\ref{t2v12}.

\begin{align*}
    \textrm{Let } P_1&=\{(i-1; j, j-1, j-2, j-3, j-4)\},    \\
    &\mbox{for } i = 1, \ldots, \frac{5m}{2} \mbox{ and } j=(15m+4)-5i.
\end{align*}

If $m=0$, $P_2$ is an empty set. If $m \geq 2$, then let 

\begin{align*}
     P_2&= \{(15m)+i; j, j-6, j-12, j-18, j-24\}   \\
    &\mbox{for } i=1, 2, \ldots, \frac{m}{2} \mbox{ and } j=(30m+30)-29i.
\end{align*}

\[\textrm{Let } L=\{(15m), (30m)\}\]

Let $D=\{1, 2, \ldots, 15m\}$, and $D_1^*=\{6k | k=1, \ldots, \frac{5m}{2}\}$. Note that $D_1^*$ is empty if $P_2$ is empty, that is when $m=0$. The stars in $P_2$ exhaust the differences in $D_1^*$. Then the stars in $P_1$ exhaust the differences in $D \setminus (D^*_1)$.

Next we construct a set of $2m$ prime stars $P_3$ simply by always choosing the next available smallest vertex for the center, and the set of next available largest vertices for the leaves. In this case, the smallest possible vertex is $(15m+2)$, and the largest possible vertex is $(30m-1)$. So, the maximum possible length of a prime edge among these leftover vertices is $(15m-3)$. However, at this point, no prime edges have appeared yet. Hence, this process guarantees that all prime edges used in $P_3$ will have distinct forward differences from the set $\{(15m-1)', (15m-2)', \ldots, 1'\}$. Thus, $P_1 \cup P_2 \cup P_3$ partition $V \setminus \{(15m), (30m)\}$. Now let $L = \{(15m), (30m)'\}$ be the prime star $L$, which contains difference $\{(15m)' \}$.

In each case, we have constructed an almost 5-star factor with two isolated vertex and the desired properties.

\end{proof}


\begin{lemma}
\label{t4}
Let $m \geq 0$. There exists an almost 5-star factor with $t=4$ on $G=\{0,1, \ldots, 30m+21\}$ where $m \geq 0$ with the following properties:
\begin{itemize}
\item Each forward difference $d \in \{1, 2, \dots, 15m+10\}$ appears at least once among the stars.
\item Each forward difference $d \in \{1, 2, \dots, 15m+10\}$ appears no more than twice among the stars.
\item There are no wrap-around edges.
\item The $4$ isolated vertices form a prime star of size $3$.
\end{itemize}

\end{lemma}

\begin{proof}
Let $V=\{0,1, \ldots, 30m+21\}$ be a set of $30m+22$ vertices.  If $m$ is odd, we give the following sets of pure stars $P_0$, $P_1$ and $P_2$; and prime stars $P_3$ on $V \setminus \{(30m+19), (30m+18),(30m+17), (15m+9)\}$.

\[\textrm{Let }P_0=\{(15m+11); (30m+21),(30m+20),(30m+14),(30m+8),(30m+2) \}, \]
\[\textrm{ and }P_1=\{(i-1; j, j-1, j-2, j-3, j-4)\},\]
\[\mbox{for } i = 1, \ldots, \frac{5m+3}{2} \mbox{ and } j=(15m+13)-5i.\]

If $m=1$, $P_2$ is an empty set. If $m \geq 3$, then let 

\[P_2= \{(15m+12)+i; j, j-6, j-12, j-18, j-24\},\]
 \[\mbox{for } i=1, 2, \ldots, \frac{m-1}{2} \mbox{ and } j=(30m+26)-29i.\]

\[\textrm{Let } L=\{(30m+19), (30m+18),(30m+17), (15m+9)\}\]

Let $D=\{1, 2, \ldots, 15m+10\}$, and $D_1^*=\{6k | k=1, \ldots, \frac{5m-5}{2}\}$. Note that $D_1^*$ is empty if $P_2$ is empty, that is when $m=1$. The stars in $P_2$ exhaust the differences in $D_1^*$. Then the stars in $P_1$ exhaust the differences in $D \setminus (D^*_1)$.

Next we construct a set of $2m+1$ prime stars $P_3$ simply by always choosing the next available smallest vertex for the center, and the set of next available largest vertices for the leaves. In this case, the smallest possible vertex is $(15m+10)$, and the largest possible vertex is $(30m+16)$. So, the maximum possible length of a prime edge among these leftover vertices is $(15m+6)$ at any $m$. However, at this point, no prime edges have appeared yet. Hence, this process guarantees that all prime edges used in $P_3$ will have distinct forward differences from the set $\{(15m+6)', (15m+5)', \ldots, 1'\}$. Thus, $P_1 \cup P_2 \cup P_3$ partition $V \setminus \{(30m+19), (30m+18),(30m+17), (15m+9)\}$. Now let $L = \{ (15m+9) ; (30m+19)', (30m+18)',(30m+17)'\}$ be the prime star $L$, which contains difference $\{(15m+10)',(15m+9)',(15m+8)' \}$.

If $m \geq 0$ is even, we give the following sets of pure stars $P_0$, $P_1$ and $P_2$; and prime stars $P_3$ on $V \setminus \{(30m+16), (30m+15),(30m+14), (15m+6)\}$.

\[\textrm{Let } P_0=\{(15m+11); (30m+21),(30m+20),(30m+19),(30m+18),(30m+17) \},\]

\[P_1=\{(i-1; j, j-1, j-2, j-3, j-4)\},\]
\[\mbox{for } i = 1, \ldots, \frac{5m+2}{2} \mbox{ and } j=(15m+10)-5i.\]

If $m=0$, $P_2$ is an empty set. If $m \geq 2$, then let 

\[P_2= \{(15m+11)+i; j, j-6, j-12, j-18, j-24\},\]
 \[\mbox{for } i=1, 2, \ldots, \frac{m}{2} \mbox{ and } j=(30m+41)-29i.\]

\[\textrm{Let } L=\{(30m+16), (30m+15),(30m+14), (15m+6)\}\]

Let $D=\{1, 2, \ldots, 15m+10\}$, and $D_1^*=\{6k | k=1, \ldots, \frac{5m}{2}\}$. Note that $D_1^*$ is empty if $P_2$ is empty, that is when $m=0$. The stars in $P_2$ exhaust the differences in $D_1^*$. Then the stars in $P_1$ exhaust the differences in $D \setminus (D^*_1)$.

Next we construct a set of $2m+1$ prime stars $P_3$ simply by always choosing the next available smallest vertex for the center, and the set of next available largest vertices for the leaves. In this case, the smallest possible vertex is $(15m+7)$, and the largest possible vertex is $(30m+13)$. So, the maximum possible length of a prime edge among these leftover vertices is $(15m+6)$ at any $m$. However, at this point, no prime edges have appeared yet. Hence, this process guarantees that all prime edges used in $P_3$ will have distinct forward differences from the set $\{(15m+6)', (15m+5)', \ldots, 1'\}$. Thus, $P_1 \cup P_2 \cup P_3$ partition $V \setminus \{(30m+16), (30m+15),(30m+14), (15m+6)\}$. Now let $L = \{(15m+6); (30m+16)', (30m+15)',(30m+14)'\}$ be the prime star $L$, which contains difference $\{(15m+10)',(15m+9)',(15m+8)' \}$.

In each case, we have constructed an almost 5-star factor with four isolated vertices and the desired properties.

\end{proof}

\section{Part I factors}

In this section, we use {\em almost 5-star factors} on a set of $v$ vertices to build $5$-star factors on a set of $6v$ vertices.

\begin{lemma} 
\label{Part I}
(Part I factors) If there exists an almost 5-star factor on $v$ vertices with $t$ isolated vertices, then there exists $v$ 5-star factors on $6v$ vertices.
\end{lemma}

\begin{proof}

{\bf Case $t=1$:} Suppose $v=30m+7$ where $m \geq 1$. Let $F$ be the almost 5-star factor constructed in Lemma~\ref{t1} on  $\{0,1, \ldots, 30m+6\}$ with the one isolated vertex, $x$. We will form a set of base blocks on $V=\{0,1, \ldots, 180m+41\}$. Let $V = \bigcup_{i=0}^5 V_i$ where $V_i=\{v \in V | v\equiv i \pmod{6} \}$. Recall that $F$ consists of $3m$ pure stars, one mixed star with three pure edges and two prime edges,  and $2m$ prime stars. 

For each pure star $s=(c; l_1, l_2, l_3, l_4, l_5) \in F$, construct the star $s_i=(6c+i; 6l_1+i, 6l_2+i, 6l_3+i, 6l_4+i, 6l_5+i)$ for $i=0,1,2,3,4,5$. The mixed star has $3$ pure edges and $2$ prime edges. For the mixed star $m=(c; l_1, l_2, l_3, l_4', l_5')$, construct the stars $m_{i}=(6c+i; 6l_{1}+i, 6l_{2}+i, 6l_{3}+i, 6l_{4}'+((i+1) \pmod{6}), 6l_{5}'+((i+2) \pmod{6}))$ for $i=0,1,2,3,4,5$. For each prime star $p=(c'; l_1', l_2', l_3', l_4', l_5')$, construct the prime stars $p_i=(6c'; 6l_1'+((1+i) \pmod{6}), 6l_2'+((2+i) \pmod{6}), 6l_3'+((3+i) \pmod{6}), 6l_4'+((4+i) \pmod{6}), 6l_5'+((5+i) \pmod{6}))$ for $i=0,1,2,3,4,5$. 

Therefore, we have formed a set of $5$-stars that spans every vertex in $V$ except the vertices in the set $\{6x, 6x+1, 6x+2, 6x+3, 6x+4, 6x+5\}$. Create one more star, $s^{\star}=(6x; 6x+1, 6x+2, 6x+3, 6x+4, 6x+5)$. These base blocks give a factor of $5$-stars on $V$ called $B_0$. For $i=1,2,\dots, 30m+6$, let $B_i =\{s+i \pmod{30m+6} : s \in B \}$. Then, $\cup_{i=0}^{30m+6}B_i$ is a set of $30m+7$ $5$-star factors on $V$.\\

{\bf Case t=3:} Suppose $v=30m+27$ for any non-negative integer $m$. Let $F$ be the almost 5-star factor constructed in Lemma~\ref{t3} on  $\{0,1, \ldots, 30m+26\}$ with the little prime star of size $2$, $L$. Let $L=\{x_1;x_2,x_3 \}$ where $x_1 < x_2 < x_3$. We will form a set of base blocks on $V=\{0,1, \ldots, 180m+161\}$. Let $V = \bigcup_{i=0}^5 V_i$ where $V_i=\{v \in V | v\equiv i \pmod{6} \}$. Recall that $F$ consists of $3m+2$ pure stars, one mixed star with $3$ pure edges and $2$ prime edges, and $2m+1$ prime stars. 

For each pure star $s=(c; l_1, l_2, l_3, l_4, l_5) \in F$, construct the stars $s_i=(6c+i; 6l_1+i, 6l_2+i, 6l_3+i, 6l_4+i, 6l_5+i)$ for $i=0,1,2,3,4,5$. The mixed star has $3$ pure edges and $2$ prime edges. For the mixed stars $m=(c; l_1, l_2, l_3, l_4', l_5')$, construct the star $m_{i}=(6c+i; 6l_{1}+i, 6l_{2}+i, 6l_{3}+i, 6l_{4}'+((i+1) \pmod{6}), 6l_{5}'+((i+2) \pmod{6}))$ for $i=0,1,2,3,4,5$. For each prime star $p=(c'; l_1', l_2', l_3', l_4', l_5')$, construct the prime stars $p_i=(6c'; 6l_1'+((1+i) \pmod{6}), 6l_2'+((2+i) \pmod{6}), 6l_3'+((3+i) \pmod{6}), 6l_4'+((4+i) \pmod{6}), 6l_5'+((5+i) \pmod{6}))$ for $i=0,1,2,3,4,5$.

Therefore, we have formed a set of $5$-stars that spans every vertex in $V$ except the vertices in the set $\{6x, 6x+1, 6x+2, 6x+3, 6x+4, 6x+5\}$ for each $x  \in \{x_1,x_2,x_3\}$. So for the little prime star $L=\{x_1;x_2,x_3 \}$, we construct three more stars; 
$s^{\star}_1=(6x_1; 6x_2+1, 6x_2+2, 6x_2+3, 6x_2+4, 6x_2+5)$,
$s^{\star}_2=(6x_1+1; 6x_3, 6x_3+2, 6x_3+3, 6x_3+4, 6x_3+5)$,
$s^{\star}_3=(6x_1+2; 6x_1+3, 6x_1+4, 6x_1+5, 6x_2, 6x_3+1)$. These base blocks give a factor of $5$-stars on $V$ called $B_0$. For $i=1,2,\dots, 30m+26$, let $B_i =\{s+i \pmod{30m+26} : s \in B \}$. Then, $\cup_{i=0}^{30m+26}B_i$ is a set of $30m+27$ $5$-star factors on $V$.\\

{\bf Case t=5:} Suppose $v=30m+17$ where $m \geq 1$. Let $F$ be the almost 5-star factor constructed in Lemma~\ref{t5} on  $\{0,1, \ldots, 30m+16\}$ with the little prime star of size $4$, $L$. Let $L=\{x_1;x_2,x_3, x_4, x_5 \}$ where $x_1 < x_2 < x_3 < x_4 < x_5$. We will form a set of base blocks on $V=\{0,1, \ldots, 180m+101\}$. Let $V = \bigcup_{i=0}^5 V_i$ where $V_i=\{v \in V | v\equiv i \pmod{6} \}$. Recall that $F$ consists of $3m+1$ pure stars, one mixed star with $3$ pure edges and $2$ prime edges, and $2m+2$ prime stars if $m$ is odd, and $2m$ prime stars if $m$ is even.

For each pure star $s=(c; l_1, l_2, l_3, l_4, l_5) \in F$, construct the stars $s_i=(6c+i; 6l_1+i, 6l_2+i, 6l_3+i, 6l_4+i, 6l_5+i)$ for $i=0,1,2,3,4,5$. The mixed star has $3$ pure edges and $2$ prime edges. For the mixed stars $m=(c; l_1, l_2, l_3, l_4', l_5')$, construct the star $m_{i}=(6c+i; 6l_{1}+i, 6l_{2}+i, 6l_{3}+i, 6l_{4}'+((i+1) \pmod{6}), 6l_{5}'+((i+2) \pmod{6}))$ for $i=0,1,2,3,4,5$. For each prime star $p=(c'; l_1', l_2', l_3', l_4', l_5')$, construct the prime stars $p_i=(6c'; 6l_1'+((1+i) \pmod{6}), 6l_2'+((2+i) \pmod{6}), 6l_3'+((3+i) \pmod{6}), 6l_4'+((4+i) \pmod{6}), 6l_5'+((5+i) \pmod{6}))$ for $i=0,1,2,3,4,5$.

Therefore, we have formed a set of $5$-stars that spans every vertex in $V$ except the vertices in the set $\{6x, 6x+1, 6x+2, 6x+3, 6x+4, 6x+5\}$ for all $x  \in \{x_1,x_2,x_3,x_4,x_5\}$. So for the little prime star $L=\{x_1;x_2,x_3,x_4,x_5\}$, we construct five more stars; 
$s^{\star}_1=(6x_1; 6x_2+1, 6x_2+2, 6x_2+3, 6x_2+4, 6x_2+5)$,
$s^{\star}_2=(6x_1+1; 6x_3, 6x_3+2, 6x_3+3, 6x_3+4, 6x_3+5)$,
$s^{\star}_3=(6x_1+2; 6x_4, 6x_4+1, 6x_4+3, 6x_4+4, 6x_4+5)$. 
$s^{\star}_4=(6x_1+3; 6x_5, 6x_5+1, 6x_5+2, 6x_5+4, 6x_5+5)$. 
$s^{\star}_5=(6x_1+4; 6x_1+5, 6x_2, 6x_3+1, 6x_4+2, 6x_5+3)$. 

These base blocks give a factor of $5$-stars on $V$ called $B_0$. For $i=1,2,\dots,$ $30m+16$, let $B_i =\{s+i \pmod{30m+16} : s \in B \}$. Then, $\cup_{i=0}^{30m+16}B_i$ is a set of $30m+17$ $5$-star factors on $V$.\\

{\bf Case t=0:} Suppose $v=30m+12$ where $m \geq 0$. Let $F$ be the 5-star factor constructed in Lemma~\ref{t0} on  $\{0,1, \ldots, 30m+11\}$. We will form a set of base blocks on $V=\{0,1, \ldots, 180m+71\}$. Let $V = \bigcup_{i=0}^5 V_i$ where $V_i=\{v \in V | v\equiv i \pmod{6} \}$. Recall that $F$ consists of $3m+1$ pure stars and $2m+1$ prime stars. 

For each pure star $s=(c; l_1, l_2, l_3, l_4, l_5) \in F$, construct the star $s_i=(6c+i; 6l_1+i, 6l_2+i, 6l_3+i, 6l_4+i, 6l_5+i)$ for $i=0,1,2,3,4,5$. For each prime star $p=(c'; l_1', l_2', l_3', l_4', l_5')$, construct the prime stars $p_i=(6c'; 6l_1'+((1+i) \pmod{6}), 6l_2'+((2+i) \pmod{6}), 6l_3'+((3+i) \pmod{6}), 6l_4'+((4+i) \pmod{6}), 6l_5'+((5+i) \pmod{6}))$ for $i=0,1,2,3,4,5$.

Therefore, we have formed a set of $5$-stars that spans every vertex in $V$. These base blocks give a factor of $5$-stars on $V$ called $B_0$. For $i=1,2,\dots, 30m+11$, let $B_i =\{s+i \pmod{30m+11} : s \in B \}$. Then, $\cup_{i=0}^{30m+11}B_i$ is a set of $30m+12$ $5$-star factors on $V$.\\

{\bf Case t=2:} Suppose $v=30m+2$ for any $m\geq1$. Let $F$ be the almost 5-star factor constructed in Lemma~\ref{t3} on  $\{0,1, \ldots, 30m+1\}$ with the edge $L=\{x_1,x_2 \}$ where $x_1 < x_2$. We will form a set of base blocks on $V=\{0,1, \ldots, 180m+12\}$. Let $V = \bigcup_{i=0}^5 V_i$ where $V_i=\{v \in V | v\equiv i \pmod{6} \}$. Recall that $F$ consists of $3m+1$ pure stars and $2m$ prime stars. 

For each pure star $s=(c; l_1, l_2, l_3, l_4, l_5) \in F$, construct the stars $s_i=(6c+i; 6l_1+i, 6l_2+i, 6l_3+i, 6l_4+i, 6l_5+i)$ for $i=0,1,2,3,4,5$. For each prime star $p=(c'; l_1', l_2', l_3', l_4', l_5')$, construct the prime stars $p_i=(6c'; 6l_1'+((1+i) \pmod{6}), 6l_2'+((2+i) \pmod{6}), 6l_3'+((3+i) \pmod{6}), 6l_4'+((4+i) \pmod{6}), 6l_5'+((5+i) \pmod{6}))$ for $i=0,1,2,3,4,5$.

Therefore, we have formed a set of $5$-stars that spans every vertex in $V$ except the vertices in the set $\{6x, 6x+1, 6x+2, 6x+3, 6x+4, 6x+5\}$ for all $x  \in \{x_1,x_2 \}$. So for the edge $L=\{x_1,x_2\}$, we construct two more stars; 
$s^{\star}_1=(6x_1; 6x_2+1, 6x_2+2, 6x_2+3, 6x_2+4, 6x_2+5)$,
$s^{\star}_2=(6x_1+1; 6x_1+2, 6x_1+3, 6x_1+4, 6x_1+5, 6x_2+1)$. These base blocks give a factor of $5$-stars on $V$ called $B_0$. For $i=1,2,\dots, 30m+1$, let $B_i =\{s+i \pmod{30m+1} : s \in B \}$. Then, $\cup_{i=0}^{30m+1}B_i$ is a set of $30m+2$ $5$-star factors on $V$.\\

{\bf Case t=4:} Suppose $v=30m+22$ for any non negative integer $m$. Let $F$ be the almost 5-star factor constructed in Lemma~\ref{t3} on  $\{0,1, \ldots, 30m+21\}$ with the little prime star of size $3$, $L$. Let $L=\{x_1;x_2,x_3,x_4\}$ where $x_1 < x_2 < x_3 < x_4$. We will form a set of base blocks on $V=\{0,1, \ldots, 180m+131\}$. Let $V = \bigcup_{i=0}^5 V_i$ where $V_i=\{v \in V | v\equiv i \pmod{6} \}$. Recall that $F$ consists of $3m+2$ pure stars and $2m+1$ prime stars. 

For each pure star $s=(c; l_1, l_2, l_3, l_4, l_5) \in F$, construct the stars $s_i=(6c+i; 6l_1+i, 6l_2+i, 6l_3+i, 6l_4+i, 6l_5+i)$ for $i=0,1,2,3,4,5$. For each prime star $p=(c'; l_1', l_2', l_3', l_4', l_5')$, construct the prime stars $p_i=(6c'; 6l_1'+((1+i) \pmod{6}), 6l_2'+((2+i) \pmod{6}), 6l_3'+((3+i) \pmod{6}), 6l_4'+((4+i) \pmod{6}), 6l_5'+((5+i) \pmod{6}))$ for $i=0,1,2,3,4,5$.

Therefore, we have formed a set of $5$-stars that spans every vertex in $V$ except the vertices in the set $\{6x, 6x+1, 6x+2, 6x+3, 6x+4, 6x+5\}$ for all $x  \in \{x_1,x_2,x_3,x_4\}$. So for the little prime star $L=\{x_1;x_2,x_3,x_4 \}$, we construct four more stars; 
$s^{\star}_1=(6x_1; 6x_2+1, 6x_2+2, 6x_2+3, 6x_2+4, 6x_2+5)$,
$s^{\star}_2=(6x_1+1; 6x_3+2, 6x_3+3, 6x_3+4, 6x_3+5, 6x_3)$,
$s^{\star}_3=(6x_1+2; 6x_4+3, 6x_4+4, 6x_4+5, 6x_4, 6x_4+1)$,
$s^{\star}_4=(6x_1+3; 6x_1+4, 6x_1+5, 6x_2, 6x_3+1, 6x_4+2)$. These base blocks give a factor of $5$-stars on $V$ called $B_0$. For $i=1,2,\dots, 30m+21$, let $B_i =\{s+i \pmod{30m+21} : s \in B \}$. Then, $\cup_{i=0}^{30m+21}B_i$ is a set of $30m+22$ $5$-star factors on $V$.\\

\end{proof}

\section{Balanced Star Arrays}

In the graph $K_v$, for each difference $d$, there are $v$ edges with that difference. So when decomposing $K_v-I$ into $5$-star factors, we must ensure that for any difference $d$, each edge with difference $d$ appears exactly once in a star. To keep track of the differences that were used in the Part I factors and the differences we still need to cover to complete the decomposition, we will use a structure called a balanced star array for each set of vertices, $V_i$, $i=0,1,\dots, 5$.\\

A {\em balanced star array}, for a set of vertices $V$, is a $\lceil \frac{v-2}{12}\rceil \times 6$ array $T=T^{1} \cup T^{2}$ whose entries partition the set $D'$ where
$D=\{0,1, \ldots, \frac{v-2}{2}\}$ and $D'=D \setminus \{d \in D | d \equiv 0 \pmod{6}\}$, and satisfies the following properties:
\begin{itemize}
\item The columns are indexed by $\{1, 2, \ldots, 5\}$, and all entries in column $j$ are $\equiv j \pmod{(6)}$.
\item $T^{1}$ is a subarray of $T$ whose entries represent the differences covered by the Part I stars.  
\item One row of $T^1$ will contain $r$ empty cells, where $r$ is the remainder when $\frac{v-2}{2}$ is divided by 6.
\item $T^{2}$ is a subarray of $T$ with no empty cells.
\end{itemize}

    Each entry $d$ in $T$ represents all edges $\{u,v\}$, with difference $d$ such that $u<v$ and $u \in V$. The entries in $T^1$ are differences from $D'$ that have been covered in the {\em Part I} factors, and the entries in $T^2$ are differences that have not yet been covered. Note that none of the differences in $D'$ are congruent to $0 \pmod{6}$. This is because every edge with difference $d \equiv 0 \pmod{6}$ is contained in exactly one pure or mixed star. Thus, we are only concerned with the differences that are covered by prime edges. We will build the arrays so that each full row of $T^1$ corresponds to the set of $5$ differences covered by a particular {\em Part I} prime star. If $\frac{v-2}{2}$ is not divisible by $6$, then one row of $T^1$ will contain $r$ empty cells, where $r$ is the remainder when $\frac{v-2}{2}$ is divisible by $6$. The non-empty cells in this row corresponds to the prime edges in the mixed star from {\em Part I}.

\begin{lemma} 
\label{Part II}
(Part II factors) If there exists a balanced star array for each set $V_{i}, i \in \mathbb{Z}_{6}$, then there is a decomposition of $K_{v}-I$ into 5-star factors.
\end{lemma}

\begin{proof}
Every edge with difference $d \equiv 0 \pmod{6}$ is contained in exactly one pure or mixed star from {\em Part I}. Therefore, we need only be concerned with ensuring that each edge with difference $d \not \equiv 0 \pmod{6}$ is contained in exactly one $5$-star. Let $V = \bigcup_{i=0}^5 V_i$, and let $T_{i}$ be the balanced star array for the set $V_{i}$. The differences in $T_{i}^{1}$ are covered by the factors given in Part I. For each row of the subarray $T_{i}^{2}$, we construct a 5-star factor as follows. Let the entries in the given row be $(d_1, d_2, d_3, d_4, d_5)$. Construct the base star to be $s=(i;i+d_1, i+d_2, i+d_3, i+d_4, i+d_5)$. We obtain $\frac{v-6}{6}$ more stars by taking $s+6j$ for $j=1,2, \ldots, \frac{v-6}{6}$. Because each $d_k \equiv k \pmod{6}$, for $k=1,2,3,4,5$,  we are guaranteed that these stars are disjoint and will span the set $V$. Furthermore, each forward edge of difference $d_k$ on the vertices of $V_{i}$ has been covered exactly once by this $5$-star factor. Because the balanced star array for $V_{i}$ partitions $D'$, we have exhausted all of edges $\{u,v\}$ with difference $d$ such that $u<v$ and $u \in V_i$. Because there is a balanced star array for each $V_{i}$, we have covered all edges of each difference. Thus we have decomposed $K_{v}-I$ into $5$-star factors.
\end{proof}

Next we build the needed balanced star arrays.\\

\begin{lemma} 
\label{1}
There is a balanced star array for each $V_{i}$, $i \in \mathbb{Z}_6$ when $v=180m+42$ with $m \geq 1 $.
\end{lemma}

\begin{proof}
We first build the array for $V_{0}$. We build the rows in $T_{0}^{1}$ based on the stars that are given in Lemma~\ref{Part I}, case $t=1$. Let one row be $[1,2,3,4,5]$, corresponding to 
the forward differences on $V_{0}$ from the star $s^{\star}$. Each prime star $p$, produces the row $[d_1,d_2,d_3,d_4,d_5]$, where $d_j=6l_{j}'+j-6c$ for $j=1,2,3,4,5$. Note 
that $d_{j} \equiv j \pmod{6}$. For the mixed star $m$, the differences covered are: $6l_1-6c$, $6l_2-6c$, $6l_3-6c$, $6l_4'+1-6c$, and $6l_5'+2-6c$. Thus the mixed star produces the row $[d_1, d_2, \phi, \phi, \phi]$ corresponding to $d_1=6l_4'+1-6c$ and $d_2=6l_5'+2-6c$, in which $d_j \equiv j \pmod{6}$ for $j=1,2$. 
Because $|D'|=75m+17 \equiv 2 \pmod{5}$, it follows that there are three sets of $15m+3$ differences, where the differences in these sets are all equivalent to $j\pmod{6}$, for $j=3,4,5$, and there are 
two sets of $15m+4$ differences, where the differences in these sets all equivalent to $j \pmod{6}$ for $j=1,2$. The differences from $D'$ that are covered by the Part I stars 
are such that there are three sets of $2m+1$ differences, where the differences in these sets are all equivalent to $j \pmod{6}$ for $j=3,4,5$. Also, there are two sets of 
$2m+2$ differences, where the differences in these sets all equivalent to $j \pmod{6}$, for $j=1,2$. This leaves five sets of $13m+2$ differences, where the differences in each set are all equivalent to $j \pmod{5}$ for $j=1,2,3,4,5$. Therefore, these remaining differences have the property that they can be partitioned into the $(13m+2) \times 5$ subarray, $T_{0}^{2}$.

Now for $i=1,2,3,4,5$, we build the balanced star array for $V_{i}$ as follows.  Beginning with $T_{i}^{1}$, for each prime star $p_{i}$, the differences covered gives the row $[d_1,d_2,d_3,d_4,d_5]$ where 

\begin{align*}
    d_{j}=&6l_{j}'+(i+j)-(6c+i)\pmod{6}\\
    \equiv& j \pmod{6}.
\end{align*}

For the mixed star $m_{i}$, the differences covered are: $6(l_1-c), 6(l_2-c), 6(l_3-c), 6l_4'+(i+1)-(6c+i), 6l_5'+(i+2)-(6c+i)$.  Create the row $[d_1,d_2,\phi,\phi,\phi]$, corresponding to $d_1=6(l_4'-c)+1$ and $d_2=6(l_5'-c)+2$, in which $d_j \equiv j \pmod{6}$ for $j=1,2$. This accounts for three sets of $2m$ differences, in which the differences in each set are equivalent to $j\pmod{6}$, for $j=3,4,5$, and two sets of $2m+1$ differences, in which the differences in each set are equivalent to $j \pmod{6}$, for $j=1,2$. This leaves a total of $65m+15$ differences which can be partitioned into the $(13m+3) \times 5$ subarray, $T_{i}^{2}$. 

\end{proof}


\begin{lemma} 
\label{5}
There is a balanced star array for each $V_{i}$, $i \in \mathbb{Z}_6$ when  $v=180m+102$ with $m \geq 1 $.
\end{lemma}
\begin{proof}
    We first build the array for $V_{i}$ where $i=0,1,2,3,4$. We build the rows in $T_{i}^{1}$ based on the stars that are given in Lemma~\ref{Part I}, case $t=5$. Let one row be $[s_1 -s_0,s_2 -s_0,s_3 -s_0,s_4 -s_0,s_5 -s_0]$, corresponding to 
the forward differences on $V_{i}$ from the star $s^{\star}_{i+1}=(s_0 ; s_1,s_2,s_3,s_4,s_5)$. Each prime star $p$, produces the row $[d_1,d_2,d_3,d_4,d_5]$, where $d_j=6l_{j}'+j-6c$ for $j=1,2,3,4,5$. Note 
that $d_{j} \equiv j \pmod{6}$. For the mixed star $m$, the differences covered are: $6l_1-6c$, $6l_2-6c$, $6l_3-6c$, $6l_4'+1-6c$, and $6l_5'+2-6c$. Thus the mixed star produces the row $[d_1, d_2, \phi, \phi, \phi]$ corresponding to $d_1=6l_4'+1-6c$ and $d_2=6l_5'+2-6c$, in which $d_j \equiv j \pmod{6}$ for $j=1,2$. Because $|D'|=75m+42 \equiv 2 \pmod{5}$,  it means that there are three sets of $15m+8$ differences, where the differences in these sets are all equivalent to $j\pmod{6}$, for $j=3,4,5$, and there are two sets of $15m+9$ differences, where the differences in these sets all equivalent to $j \pmod{6}$ for $j=1,2$. The differences from $D'$ that are covered by the Part I stars are such that there are three sets of $2m+3$ differences if $m$ is odd and $2m+1$ differences if $m$ is even, where the differences in these sets are all equivalent to $j \pmod{6}$ for $j=3,4,5$. Also, there are two sets of $2m+4$ differences if $m$ is odd and $2m+2$ differences if $m$ is even, where the differences in these sets all equivalent to $j \pmod{6}$, for $j=1,2$. This leaves five sets of $65m+25$ differences when $m$ is odd and $65m+35$ differences when $m$ is even, where the differences in each set are all equivalent to $j \pmod{5}$ for $j=1,2,3,4,5$. Therefore, these remaining differences have the property that they can be partitioned into the $(13m+7) \times 5$ subarray, $T_{i}^{2}$.

Now for $i=5$, we build the balanced star array for $V_{i}$ as follows.  Beginning with $T_{i}^{1}$, for each prime star $p_{i}$, the differences covered gives the row $[d_1,d_2,d_3,d_4,d_5]$ where 

\begin{align*}
    d_{j}=&6l_{j}'+(i+j)-(6c+i)\pmod{6}\\
    \equiv& j \pmod{6}.
\end{align*}

For the mixed star $m_{i}$, the differences covered are: $6(l_1-c), 6(l_2-c), 6(l_3-c), 6l_4'+(i+1)-(6c+i), 6l_5'+(i+2)-(6c+i)$.  Create the row $[d_1,d_2,\phi,\phi,\phi]$, corresponding to $d_1=6(l_4'-c)+1$ and $d_2=6(l_5'-c)+2$, in which $d_j \equiv j \pmod{6}$ for $j=1,2$. 

If $m$ is odd, this accounts for three sets of $2m+2$ differences in which the differences in each set are equivalent to $j\pmod{6}$ for $j=3,4,5$, and two sets of $2m+3$ differences in which the differences in each set are equivalent to $j \pmod{6}$, for $j=1,2$. This leaves a total of $65m+30$ differences if $m$ is odd which can be partitioned into the $(13m+6) \times 5$ subarray, $T_{i}^{2}$. 

If $m$ is even, this accounts for three sets of $2m$ differences in which the differences in each set are equivalent to $j\pmod{6}$ for $j=3,4,5$, and two sets of $2m+1$ differences in which the differences in each set are equivalent to $j \pmod{6}$, for $j=1,2$. This leaves a total of $65m+40$ differences if $m$ is even which can be partitioned into the $(13m+8) \times 5$ subarray, $T_{i}^{2}$. 
\end{proof}


\begin{lemma}
\label{3}
There is a balanced star array for each $V_{i}$, $i \in \mathbb{Z}_6$ when $v=180m+162$  with $m \geq 0 $.
\end{lemma}

\begin{proof}
    We first build the array for $V_{i}$ where $i=0,1,2$. We build the rows in $T_{i}^{1}$ based on the stars that are given in Lemma~\ref{Part I}, Case $t=3$.  
    
 Let one row be $[s_1 -s_0,s_2 -s_0,s_3 -s_0,s_4 -s_0,s_5 -s_0]$, corresponding to 
the forward differences on $V_{i}$ from the star $s^{\star}_{i+1}=(s_0 ; s_1,s_2,s_3,s_4,s_5)$. Each prime star $p$, produces the row $[d_1,d_2,d_3,d_4,d_5]$, where $d_j=6l_{j}'+j-6c$ for $j=1,2,3,4,5$. Note 
that $d_{j} \equiv j \pmod{6}$. For the mixed star $m$, the differences covered are: $6l_1-6c$, $6l_2-6c$, $6l_3-6c$, $6l_4'+1-6c$, and $6l_5'+2-6c$. Thus the mixed star produces the row $[d_1, d_2, \phi, \phi, \phi]$ corresponding to $d_1=6l_4'+1-6c$ and $d_2=6l_5'+2-6c$, in which $d_j \equiv j \pmod{6}$ for $j=1,2$. Because $|D'|=75m+67 \equiv 2 \pmod{5}$,  it means that there are three sets of $15m+13$ differences, where the differences in these sets are all equivalent to $j\pmod{6}$, for $j=3,4,5$, and there are 
two sets of $15m+14$ differences, where the differences in these sets all equivalent to $j \pmod{6}$ for $j=1,2$. The differences from $D'$ that are covered by the Part I stars are such that there are three sets of $2m+2$ differences, where the differences in these sets are all equivalent to $j \pmod{6}$ for $j=3,4,5$. Also, there are two sets of 
$2m+3$ differences, where the differences in these sets all equivalent to $j \pmod{6}$, for $j=1,2$. This leaves five sets of $13m+11$ differences, where the differences in each set are all equivalent to $j \pmod{5}$ for $j=1,2,3,4,5$. Therefore, these remaining differences have the property that they can be partitioned into the $(13m+11) \times 5$ subarray, $T_{0}^{2}$.

Now for $i=3,4,5$, we build the balanced star array for $V_{i}$ as follows.  Beginning with $T_{i}^{1}$, for each prime star $p_{i}$, the differences covered gives the row $[d_1,d_2,d_3,d_4,d_5]$ where 

\begin{align*}
    d_{j}=&6l_{j}'+(i+j)-(6c+i)\pmod{6}\\
    \equiv& j \pmod{6}.
\end{align*}

For the mixed star $m_{i}$, the differences covered are: $6(l_1-c), 6(l_2-c), 6(l_3-c), 6l_4'+(i+1)-(6c+i), 6l_5'+(i+2)-(6c+i)$.  Create the row $[d_1,d_2,\phi,\phi,\phi]$, corresponding to $d_1=6(l_4'-c)+1$ and $d_2=6(l_5'-c)+2$, in which $d_j \equiv j \pmod{6}$ for $j=1,2$. This accounts for three sets of $2m+1$ differences, in which the differences in each set are equivalent to $j\pmod{6}$, for $j=3,4,5$, and two sets of $2m+2$ differences, in which the differences in each set are equivalent to $j \pmod{6}$, for $j=1,2$. This leaves a total of $65m+60$ differences which can be partitioned into the $(13m+12) \times 5$ subarray, $T_{i}^{2}$. 

\end{proof}

\begin{lemma} 
\label{0}
There is a balanced star array for each $V_{i}$, $i \in \mathbb{Z}_6$ when $v=180m+72$ with $m \geq 0$.
\end{lemma}

\begin{proof}
We first build the array for $V_{0}$. We build the rows in $T_{0}^{1}$ based on the stars that are given in Lemma~\ref{Part I}, case $t=0$. Let one row be $[1,2,3,4,5]$, corresponding to 
the forward differences on $V_{0}$ from the star $s^{\star}$. Each prime star $p$, produces the row $[d_1,d_2,d_3,d_4,d_5]$, where $d_j=6l_{j}'+j-6c$ for $j=1,2,3,4,5$. Note 
that $d_{j} \equiv j \pmod{6}$. Here we have $|D'|=75m+30 \equiv 0 \pmod{5}$, so it follows that there are five sets of $15m+6$ differences. The differences from $D'$ that are covered by the Part I stars 
are such that there are five sets of $2m+1$ differences. This leaves five sets of $13m+5$ differences, where the differences in each set are all equivalent to $j \pmod{5}$ for $j=1,2,3,4,5$. Therefore, these remaining differences have the property that they can be partitioned into the $(13m+5) \times 5$ subarray, $T_{0}^{2}$.

Now for $i=1,2,3,4,5$, we build the balanced star array for $V_{i}$ as follows. Beginning with $T_{i}^{1}$, for each prime star $p_{i}$, the differences covered gives the row $[d_1,d_2,d_3,d_4,d_5]$ where 
\begin{align*}
    d_{j}=&6l_{j}'+(i+j)-(6c+i)\pmod{6}\\
    \equiv& j\pmod{6}.
\end{align*}

This leaves a total of $65m+25$ differences which can be partitioned into the $(13m+5) \times 5$ subarray, $T_{i}^{2}$. 

\end{proof}

\begin{lemma} 
\label{2}
There is a balanced star array for each $V_{i}$, $i \in \mathbb{Z}_6$ when $v=180m+12$ with $m \geq 1$.
\end{lemma}

\begin{proof}
We first build the array for $V_{0}$. We build the rows in $T_{0}^{1}$ based on the stars that are given in Lemma~\ref{Part I}, case $t=2$. Let one row be $[1,2,3,4,5]$, corresponding to 
the forward differences on $V_{0}$ from the star $s^{\star}$. Each prime star $p$, produces the row $[d_1,d_2,d_3,d_4,d_5]$, where $d_j=6l_{j}'+j-6c$ for $j=1,2,3,4,5$. Note 
that $d_{j} \equiv j \pmod{6}$. Here we have $|D'|=75m+5 \equiv 0 \pmod{5}$, it follows that there are five sets of $15m+1$ differences. The differences from $D'$ that are covered by the Part I stars 
are such that there are five sets of $2m$ differences. This leaves five sets of $13m+1$ differences, where the differences in each set are all equivalent to $j \pmod{5}$ for $j=1,2,3,4,5$. Therefore, these remaining differences have the property that they can be partitioned into the $(13m+1) \times 5$ subarray, $T_{0}^{2}$.

Now for $i=1,2,3,4,5$, we build the balanced star array for $V_{i}$ as follows. Beginning with $T_{i}^{1}$, for each prime star $p_{i}$, the differences covered gives the row $[d_1,d_2,d_3,d_4,d_5]$ where 
\begin{align*}
    d_{j}=&6l_{j}'+(i+j)-(6c+i)\pmod{6}\\
    \equiv& j\pmod{6}.
\end{align*}

This leaves a total of $65m+5$ differences which can be partitioned into the $(13m+1) \times 5$ subarray, $T_{i}^{2}$. 

\end{proof}

\begin{lemma} 
\label{4}
There is a balanced star array for each $V_{i}$, $i \in \mathbb{Z}_6$ when $v=180m+132$ with $m \geq 1$.
\end{lemma}

\begin{proof}
We first build the array for $V_{0}$. We build the rows in $T_{0}^{1}$ based on the stars that are given in Lemma~\ref{Part I}, case $t=4$. Let one row be $[1,2,3,4,5]$, corresponding to 
the forward differences on $V_{0}$ from the star $s^{\star}$. Each prime star $p$, produces the row $[d_1,d_2,d_3,d_4,d_5]$, where $d_j=6l_{j}'+j-6c$ for $j=1,2,3,4,5$. Note 
that $d_{j} \equiv j \pmod{6}$. Here we have $|D'|=75m+55 \equiv 0 \pmod{5}$, it follows that there are five sets of $15m+11$ differences. The differences from $D'$ that are covered by the Part I stars 
are such that there are five sets of $2m+1$ differences. This leaves five sets of $13m+10$ differences, where the differences in each set are all equivalent to $j \pmod{5}$ for $j=1,2,3,4,5$. Therefore, these remaining differences have the property that they can be partitioned into the $(13m+10) \times 5$ subarray, $T_{0}^{2}$.

Now for $i=1,2,3,4,5$, we build the balanced star array for $V_{i}$ as follows. Beginning with $T_{i}^{1}$, for each prime star $p_{i}$, the differences covered gives the row $[d_1,d_2,d_3,d_4,d_5]$ where 
\begin{align*}
    d_{j}=&6l_{j}'+(i+j)-(6c+i)\pmod{6}\\
    \equiv& j\pmod{6}.
\end{align*}

This leaves a total of $65m+50$ differences which can be partitioned into the $(13m+10) \times 5$ subarray, $T_{i}^{2}$. 

\end{proof}

\section{Results}

We begin by giving some direct results.

\begin{lemma}
\label{t1v42}
Let $v=42$. There is a decomposition of $K_v-I$ into 5-star factors.
\end{lemma}

\begin{proof}
       Let $V=\{0,1,\dots,41 \}$ be the vertex set, and $V_i = \{v \in V | v \equiv i \pmod{6}\}$. Let $I = \{(i,i+21) : i \in \{0,1,2,\dots, 20 \} \} $ be the $1$-factor, and we give a decomposition of $K_v-I$ as follows. Let $F = \{s_0,s_1,s_2,s_3,s_4,s_5,s_6 \}$ be a $5$-star factor with stars $s_i$ where $i=0,1,\dots,6$. For each $s_i$, we let $l_{s_i}$ be the set of differences covered by $s_i$. If any edge is a wrap-around edge, then we denote its difference $l$ by $\bar l$.
\begin{align*}
    F : s_0 =& \{36; 37, 38, 39, 40, 41\}, l_{s_0} = \{ 1, 2, 3, 4, 5 \} \\
    s_1 =& \{0 ; 6, 12, 18, \overline{28}, \overline{35}\}, l_{s_1} = \{ 6, 12, 18, \overline{14}, \overline{7}\}  \\
    s_2 =& \{1 ; 7, 13, 19, \overline{29}, \overline{30}\}, l_{s_2} = \{ 6, 12, 18, \overline{13}, \overline{14}\}   \\
    s_3 =& \{2 ; 8, 14, 20, \overline{24}, \overline{31}\}, l_{s_3} = \{ 6, 12, 18, \overline{20}, \overline{13}\}   \\
    s_4 =& \{3 ; 9, 15, 21, \overline{25}, \overline{32}\}, l_{s_4} = \{ 6, 12, 18, \overline{20}, \overline{13}\}   \\
    s_5 =& \{4 ; 10, 16, 22, \overline{26}, \overline{33}\}, l_{s_5} = \{ 6, 12, 18, \overline{20}, \overline{13}\}   \\
    s_6 =& \{5 ; 11, 17, 23, \overline{27}, \overline{34}\}, l_{s_6} = \{ 6, 12, 18, \overline{20}, \overline{13}\}       
\end{align*}

We record the differences used in balanced star arrays, which are given in Figure~\ref{v42pic}. Let $T_i$ denote the balanced star array for $V_i$, $i\in \mathbb{Z}_6$. Then, by Lemma~\ref{Part II}, there is a decomposition of $K_{42}-I$ into 5-star factors.

\begin{figure}[!htb]

    \begin{minipage}{.5\linewidth}
        \centering  
        \begin{tabular}{|c|ccccc|}
        \hline
         &  &  & $T_0$ &  & \\
         \hline
        $T_0^1$ & 1 & 2 & 3 & 4 & 5 \\
         & 13 & 20 & * & * & * \\
         \hline
        $T_0^2$ & 7 & 8 & 9 & 10 & 11 \\
         & 19 & 14 & 15 & 16 & 17\\
         &  &  &  &  & \\
        \hline
       \end{tabular}
    \end{minipage}
    \begin{minipage}{.5\linewidth}
    \begin{tabular}{|c|ccccc|}
        \hline
         &  &  & $T_1$ &  & \\
         \hline
        $T_1^1$ & 13 & 20 & * & * & * \\
         &  &  &  &  &  \\
         \hline
        $T_1^2$ & 1 & 2 & 3 & 4 & 5 \\
         & 7 & 8 & 9 & 10 & 11\\
         & 19 & 14 & 15 & 16 & 17 \\
        \hline
    \end{tabular}
    \end{minipage}

\begin{align*}
\end{align*}

    \begin{minipage}{.5\linewidth}
        \centering
        \begin{tabular}{|c|ccccc|}
        \hline
         &  &  & $T_2$ &  & \\
         \hline
        $T_2^1$ & 13 & 20 & * & * & * \\
         &  &  &  &  &  \\
         \hline
        $T_2^2$ & 1 & 2 & 3 & 4 & 5 \\
         & 7 & 8 & 9 & 10 & 11\\
         & 19 & 14 & 15 & 16 & 17 \\
        \hline
       \end{tabular}
    \end{minipage}
    \begin{minipage}{.5\linewidth}
    \begin{tabular}{|c|ccccc|}
        \hline
         &  &  & $T_3$ &  & \\
         \hline
        $T_3^1$ & 13 & 20 & * & * & * \\
         &  &  &  &  &  \\
         \hline
        $T_3^2$ & 1 & 2 & 3 & 4 & 5 \\
         & 7 & 8 & 9 & 10 & 11\\
         & 19 & 14 & 15 & 16 & 17 \\
        \hline
    \end{tabular}
    \end{minipage}

\begin{align*}
\end{align*}

    \begin{minipage}{.5\linewidth}
        \centering
        \begin{tabular}{|c|ccccc|}
        \hline
         &  &  & $T_4$ &  & \\
         \hline
        $T_4^1$ & 13 & 20 & * & * & * \\
         &  &  &  &  &  \\
         \hline
        $T_4^2$ & 1 & 2 & 3 & 4 & 5 \\
         & 7 & 8 & 9 & 10 & 11\\
         & 19 & 14 & 15 & 16 & 17 \\
        \hline
       \end{tabular}
    \end{minipage}
    \begin{minipage}{.5\linewidth}
    \begin{tabular}{|c|ccccc|}
        \hline
         &  &  & $T_5$ &  & \\
         \hline
        $T_5^1$ & 7 & 14 & * & * & * \\
         &  &  &  &  &  \\
         \hline
        $T_5^2$ & 1 & 2 & 3 & 4 & 5 \\
         & 13 & 8 & 9 & 10 & 11\\
         & 19 & 20 & 15 & 16 & 17 \\
        \hline
    \end{tabular}
    \end{minipage}

    \caption{Balanced star array for $v=42$}
    \label{v42pic}
\end{figure}

\end{proof}

\begin{lemma}
\label{t5v102}
Let $v=102$. There is a decomposition of $K_v-I$ into 5-star factors.
\end{lemma}

\begin{proof}
        Let $V=\{0,1,\dots,101 \}$, $V_i = \{v \in V | v \equiv i \pmod{6}\}$ with $i \in \mathbb{Z}_6$, and let $I = \{(i,i+51) : i \in \{0,1,2,\dots, 50 \} \}$ be the $1$-factor. We give $F = F_1 \cup F_2 \cup F_3$, which is a $5$-star factor with $17$ $5$-stars. For each star $s_i$, let $l_{s_i}$ be the set of differences covered by $s_i$. \\

        If any edge is a wrap-around edge, we denote its difference $l$ as $\bar l$. We give $F_1=\{s_1,s_2,s_3,s_4,s_5,s_6 \}$ as follow.
\begin{align*}
    s_1 =& \{0 ; 6, 12, 18, 24, 30\}, l_{s_1} = \{ 6, 12, 18, 24, 30\}  \\
    s_2 =& \{1 ; 7, 13, 19, 25, 31\}, l_{s_2} = \{ 6, 12, 18, 24, 30\}  \\
    s_3 =& \{2 ; 8, 14, 20, 26, 32\}, l_{s_3} = \{ 6, 12, 18, 24, 30\}  \\
    s_4 =& \{3 ; 9, 15, 21, 27, 33\}, l_{s_4} = \{ 6, 12, 18, 24, 30\}  \\
    s_5 =& \{4 ; 10, 16, 22, 28, 34\}, l_{s_5} = \{ 6, 12, 18, 24, 30\}  \\
    s_6 =& \{5 ; 11, 17, 23, 29, 35\}, l_{s_6} = \{ 6, 12, 18, 24, 30\}   
\end{align*}

We give $F_2=\{s_1,s_2,s_3,s_4,s_5,s_6 \}$ as follows.
\begin{align*}
    s_1 =& \{36 ; 72, 78, 84, \overline{94}', \overline{101}'\}, l_{s_1} = \{ 36, 42, 48,  \overline{37}',  \overline{44}'\}  \\
    s_2 =& \{37 ; 73, 79, 85, \overline{95}', \overline{96}'\}, l_{s_2} = \{ 36, 42, 48,  \overline{43}',  \overline{44}'\}  \\
    s_3 =& \{38 ; 74, 80, 86, \overline{90}', \overline{97}'\}, l_{s_3} = \{ 36, 42, 48,  \overline{43}',  \overline{50}'\}  \\
    s_4 =& \{39 ; 75, 81, 87, \overline{91}', \overline{98}'\}, l_{s_4} = \{ 36, 42, 48,  \overline{43}',  \overline{50}'\}  \\
    s_5 =& \{40 ; 76, 82, 88, \overline{92}', \overline{99}'\}, l_{s_5} = \{ 36, 42, 48,  \overline{43}',  \overline{50}'\}  \\
    s_6 =& \{41 ; 77, 83, 89, \overline{93}', \overline{100}'\}, l_{s_6} = \{ 36, 42, 48,  \overline{43}',  \overline{50}'\}  
\end{align*}

We give $F_3=\{s_1,s_2,s_3,s_4,s_5 \}$ as follows.
\begin{align*}
    s_1 =& \{42 ; 49', 50', 51', 52', 53'\}, l_{s_1} = \{  7', 8', 9', 10', 11'  \}  \\
    s_2 =& \{43 ; 54', 56', 57', 58', 59'\}, l_{s_2} = \{  11', 13', 14', 15', 16' \}  \\
    s_3 =& \{44 ; 60', 61', 63', 64', 65'\}, l_{s_3} = \{  16', 17', 19', 20', 21' \}  \\
    s_4 =& \{45 ; 66', 67', 68', 70', 71'\}, l_{s_4} = \{  21', 22', 23', 25', 26' \}  \\
    s_5 =& \{46 ; 47', 48', 55', 62', 69'\}, l_{s_5} = \{  1', 2', 9', 16', 23' \}    
\end{align*}

We record the differences covered by $F$ in balanced star arrays, which are given in Figure~\ref{v102pic}. Let $T_i$ denote the balanced star array for $V_i$, $i\in \mathbb{Z}_6$. Then, by Lemma~\ref{Part II}, there is a decomposition of $K_{102}-I$ into $5$-star factors.

\begin{figure}[!htb]

    \begin{minipage}{.5\linewidth}
        \centering  
        \begin{tabular}{|c|ccccc|}
        \hline
         &  &  & $T_0$ &  & \\
         \hline
        $T_0^1$ & 7 & 8 & 9 & 10 & 11 \\
         & 43 & 50 & * & * & * \\
         \hline
        $T_0^2$ & 1 & 2 & 3 & 4 & 5 \\
         & 13 & 14 & 15 & 16 & 17\\
         & 19 & 20 & 21 & 22 & 23\\
         & 25 & 26 & 27 & 28 & 29 \\
         & 31 & 32 & 33 & 34 & 35 \\
         & 37 & 38 & 39 & 40 & 41 \\
         & 49 & 44 & 45 & 46 & 47 \\
         \hline
       \end{tabular}
    \end{minipage}
    \begin{minipage}{.5\linewidth}
    \begin{tabular}{|c|ccccc|}
        \hline
         &  &  & $T_1$ &  & \\
         \hline
        $T_1^1$ & 13 & 14 & 15 & 16 & 11 \\
         & 43 & 50 & * & * & * \\
         \hline
        $T_1^2$ & 1 & 2 & 3 & 4 & 5 \\
         & 7 & 8 & 9 & 10 & 17\\
         & 19 & 20 & 21 & 22 & 23\\
         & 25 & 26 & 27 & 28 & 29 \\
         & 31 & 32 & 33 & 34 & 35 \\
         & 37 & 38 & 39 & 40 & 41 \\
         & 49 & 44 & 45 & 46 & 47 \\
        \hline
    \end{tabular}
    \end{minipage}

\begin{align*}
\end{align*}

    \begin{minipage}{.5\linewidth}
        \centering  
        \begin{tabular}{|c|ccccc|}
        \hline
         &  &  & $T_2$ &  & \\
         \hline
        $T_2^1$ & 19 & 20 & 21 & 16 & 17 \\
         & 43 & 50 & * & * & * \\
         \hline
        $T_2^2$ & 1 & 2 & 3 & 4 & 5 \\
         & 7 & 8 & 9 & 10 & 11\\
         & 13 & 14 & 15 & 22 & 23\\
         & 25 & 26 & 27 & 28 & 29 \\
         & 31 & 32 & 33 & 34 & 35 \\
         & 37 & 38 & 39 & 40 & 41 \\
         & 49 & 44 & 45 & 46 & 47 \\
         \hline
       \end{tabular}
    \end{minipage}
    \begin{minipage}{.5\linewidth}
    \begin{tabular}{|c|ccccc|}
        \hline
         &  &  & $T_3$ &  & \\
         \hline
        $T_3^1$ & 25 & 26 & 21 & 22 & 23 \\
         & 43 & 50 & * & * & * \\
         \hline
        $T_3^2$ & 1 & 2 & 3 & 4 & 5 \\
         & 7 & 8 & 9 & 10 & 11\\
         & 13 & 14 & 15 & 16 & 17\\
         & 19 & 20 & 27 & 28 & 29 \\
         & 31 & 32 & 33 & 34 & 35 \\
         & 37 & 38 & 39 & 40 & 41 \\
         & 49 & 44 & 45 & 46 & 47 \\
        \hline
    \end{tabular}
    \end{minipage}

\begin{align*}
\end{align*}

    \begin{minipage}{.5\linewidth}
        \centering  
        \begin{tabular}{|c|ccccc|}
        \hline
         &  &  & $T_4$ &  & \\
         \hline
        $T_4^1$ & 1 & 2 & 9 & 16 & 23 \\
         & 43 & 44 & * & * & * \\
         \hline
        $T_4^2$ & 7 & 8 & 3 & 4 & 5 \\
         & 13 & 14 & 15 & 10 & 11\\
         & 19 & 20 & 21 & 22 & 17\\
         & 25 & 26 & 27 & 28 & 29 \\
         & 31 & 32 & 33 & 34 & 35 \\
         & 37 & 38 & 39 & 40 & 41 \\
         & 49 & 50 & 45 & 46 & 47 \\
         \hline
       \end{tabular}
    \end{minipage}
    \begin{minipage}{.5\linewidth}
    \begin{tabular}{|c|ccccc|}
        \hline
         &  &  & $T_5$ &  & \\
         \hline
        $T_5^1$ & 37 & 44 & * & * & * \\
         &  & &  &  &  \\
         \hline
        $T_5^2$ & 1 & 2 & 3 & 4 & 5 \\
         & 7 & 8 & 9 & 10 & 11\\
         & 13 & 14 & 15 & 16 & 17\\
         & 19 & 20 & 21 & 22 & 23\\
         & 25 & 26 & 27 & 28 & 29 \\
         & 31 & 32 & 33 & 34 & 35 \\
         & 43 & 38 & 39 & 40 & 41 \\
         & 49 & 50 & 45 & 46 & 47 \\
        \hline
    \end{tabular}
    \end{minipage}

    \caption{Balanced star array for $v=102$}
    \label{v102pic}
\end{figure}

\end{proof}

\begin{lemma}
\label{t2v12}
Let $v=12$. There is a decomposition of $K_v-I$ into 5-star factors.
\end{lemma}

\begin{proof}
       Let $V=\{0,1,\dots,11 \}$ be the vertex set, and let $I = \{ \{0,6\} , \{1,7\} , \{2,8\} ,$
       $ \{3,9\} , \{4,10\} , \{5,11\} \}$. 
\begin{align*}
    \textrm{Let } F_1 =& \{ \{0; 1, 2, 3, 4, 5\}, \{6; 7, 8, 9, 10, 11\}  \} \\
    F_2 =& \{ \{1; 2, 3, 4, 5, 6\}, \{7; 8, 9, 10, 11, 0\}  \}  \\
    F_3 =& \{ \{2; 3, 4, 5, 6, 7\}, \{8; 9, 10, 11, 0, 1\}  \}  \\
    F_4 =& \{ \{3; 4, 5, 6, 7, 8\}, \{9; 10, 11, 0, 1, 2\}  \}  \\
    F_5 =& \{ \{4; 5, 6, 7, 8, 9\}, \{10; 11, 0, 1, 2, 3\}  \}  \\
    F_6 =& \{ \{5; 6, 7, 8, 9, 10\}, \{11; 0, 1, 2, 3, 4\}  \}      
\end{align*}

    Then, $F= \cup_{i=1}^6F_i$ gives the desired decomposition.

\end{proof}

We are now in a position to prove the main theorem.

\begin{thm} There exists a decomposition of $K_v-I$ into $5$-star factors if and only if $v \equiv 12 \pmod{30}$.
\end{thm}

\begin{proof}
Lemma~\ref{ness} gives the necessary conditions. If $v \equiv 12 \pmod{30}$, then $v \equiv 12,42,72,102,132,$ or $162 \pmod{180}$. The decompositions of $K_v$ for $v=12,42,$ and $102$ are given in Lemma~\ref{t2v12}, Lemma~\ref{t1v42}, and Lemma~\ref{t5v102} respectively. Let $v=30m+12$, and let $I=\{ \{u,v\} : D\{u,v\}=\frac{v}{2} \}$. If $m \geq 0 $ for $v \equiv 72,132,$ or $162 \pmod{180}$ and if $m \geq 1$ for $v \equiv 12,42,$ or $102 \pmod{180}$, then by Lemmas~\ref{t1} $\sim$ Lemma~\ref{t4}, there exists an almost 5-star factor with $t$ isolated vertices on $G = \{0,1,2,\dots, \frac{v}{6}-1 \}$, where $t \equiv \frac{v}{6} \pmod{6}$. Therefore, by Lemma~\ref{Part I}, there exists $v$ $5$-star factors on $v$ vertices. By Lemmas~\ref{1} $\sim$ Lemma~\ref{4}, there is a balanced star array for each $V_i$, $i \in \mathbb{Z}_6$. Thus by by Lemma~\ref{Part II}, the remaining edges of $K_v-I$ can be decomposed into $5$-star factors.

\end{proof}

We believe that the technique of using balanced star arrays will be helpful when considering decompositions of $K_v-I$ into $n$-star factors for $n>5$.

\end{document}